\documentclass{amsart}
\usepackage{amsfonts}
\usepackage{amsmath,amssymb}	
\usepackage{amsthm}
\usepackage{amscd}
\usepackage{graphics}
\usepackage{graphicx}
{
\theoremstyle{definition}
\newtheorem{Def}{Definition}

\newtheorem{Rem}{Remark}
\newtheorem*{Prob}{Problem}
}

\newtheorem{Thm}{Theorem}

\begin{document}
\title[Reeb graphs of smooth functions on closed or open manifolds]{On Reeb graphs induced from smooth functions on closed or open manifolds}
\dedicatory{Dedicated to Masaaki Umehara and Kotaro Yamada on the occation of their 60-th birthdays}
\author{Naoki Kitazawa}
\address{Institute of Mathematics for Industry, Kyushu University, 744 Motooka, Nishi-ku Fukuoka 819-0395, Japan}
\email{n-kitazawa@imi.kyushu-u.ac.jp}
\urladdr{https://naokikitazawa.github.io/NaokiKitazawa.html}\maketitle
\begin{abstract}
For a smooth function of a suitable class, the space of all connected components of preimages is the graph and called the {\it Reeb graph}. Reeb graphs are fundamental tools in studying algebraic topological properties and differential topological ones for Morse functions and more general functions which are not so wild. They are strong tools not only in geometry, but also in applications of mathematics such as visualizations. 

In the present paper, we study whether we can construct a smooth function with good geometric properties inducing a given graph as the Reeb graph and having prescribed preimages. This paper concentrates on smooth functions on surfaces and manifolds which may be non-closed with no boundary as a pioneering case and give answers with new ideas. This problem was essentially launched by Sharko in 2000s and various answers on functions on closed manifolds have been given by others and the author.
\end{abstract}

\section{Introduction and some terminologies and notation.}
\label{sec:1}

The {\it Reeb graph} or a {\it Kronrod-Reeb graph} of a smooth function of a suitable class on a smooth manifold is the graph obtained as the space of all connected components of preimages such that the vertex set coincides with the set of all connected components of preimages containing {\it some singular points} (see \cite{reeb} and \cite{sharko} for example). We review fundamental notions on graphs
later and after that we introduce the {\it Reeb space} and the {\it Reeb graph} in Definition \ref{def:1} rigorously.

A {\it singular} point of a smooth map is a point at which the rank of the differential is smaller than both the dimensions of the manifolds of the domain and the target. A {\it regular value} is a point in the manifold of the target whose preimage has no singular points.

For so-called {\it Morse} functions, functions with finitely many singular points on closed manifolds and functions of several suitable classes, we have Reeb graphs (\cite{saeki2021}). 

Reeb graphs are fundamental and important tools in algebraic topological theory and differential topological theory of Morse functions and their generalizations. They play important roles in geometry of manifolds. They are also strong tools in applications of mathematics such as visualizations recently.

In the present paper, we attack the following fundamental, natural and important problem.

\begin{Prob}
Can we construct a smooth function with good geometric properties inducing a given graph as the Reeb graph? Especially, can we have a function with mild singularities and prescribed preimages?
\end{Prob}

This problem has been first considered by Sharko (\cite{sharko}). Related to this pioneering work, several works have been done: the work of J. Martinez-Alfaro, I. S. Meza-Sarmiento and R. Oliveira (\cite{martinezalfaromezasarmientooliveira}), the work of Masumoto and Saeki (\cite{masumotosaeki}) and recent works such as \cite{batistacostamezasarmiento}, \cite{michalak}, \cite{michalak2} and \cite{saeki2021}. In these studies, explicit smooth functions inducing the given graphs as the Reeb graphs have been constructed. Most of the functions are ones on closed surfaces or Morse functions on closed manifolds such that preimages of regular values are disjoint unions of circles or standard spheres. In other words, preimages had been essentially assumed to be disjoint unions of spheres before the author obtained a related result of a new type in \cite{kitazawa2}, answering Problem as a pioneering work respecting preimages. \cite{kitazawa} is also closely related. 

We introduce several terminologies and notation before introducing the previous result.
Hereafter, ${\mathbb{R}}^k$ denotes the $k$-dimensional Euclidean space with the standard Euclidean metric and $||x||\geq 0$ denotes the distance between $x \in {\mathbb{R}}^k$ and the origin $0$ or equivalently, the value of the standard Euclidean norm at the vector $x \in {\mathbb{R}}^n$. This is a smooth manifold and also a Riemannian manifold where the metric is the standard Euclidean metric. $\mathbb{R}$ is for ${\mathbb{R}}^k=\mathbb{R}$.
The $k$-dimensional unit sphere is denoted by $S^k:=\{x \in {\mathbb{R}}^{k+1}\mid ||x||=1\}$ and the $k$-dimensional unit disk is denoted by $D^k:=\{x \in {\mathbb{R}}^{k}\mid ||x|| \leq 1\}$. They are smooth manifolds. We also call manifolds diffeomorphic to them a $k$-dimensional {\it standard sphere} and a $k$-dimensional {\it standard disk}, respectively.
Note also for example that $S^0$ is also a two-point set endowed with the discrete topology.

We introduce {\it Reeb spaces} and {\it Reeb graphs} of functions and maps.

Before this, we introduce several terminologies and notions on graphs. 
A {\it graph} $G$ is a $1$-dimensional simplicial complex consisting of $1$-cells and $0$-cells. The set of all $1$-cells is {\it the edge set} of the graph. Each 1-cell is called an {\it edge}. The set of all $0$-cells is {\it the vertex set} of the graph. Each 0-cell is called a {\it vertex}. 
 
Two graphs are {\it isomorphic} if there exists a homeomorphism mapping the vertex set of a graph onto the vertex set of another graph and this homeomorphism is called an {\it isomorphism} between the graphs where the graphs are topologized in canonical ways. 

A graph is {\it connected} if it is connected as a topological space endowed with the canonical topology. A {\it loop} is an edge which is homeomorphic to a circle in the graph. In the present paper, we consider connected graphs with no loops essentially.

\begin{Def}
\label{def:1}
Let a smooth manifold $X$ and a function $c:X \rightarrow \mathbb{R}$ be given. Let ${\sim}_c$ be the equivalence relation on $X$ defined by the following rule: $x_1 {\sim}_c x_2$ holds if and only if they are in a same connected component of a preimage $c^{-1}(y)$. We call the quotient space $W_c:=X/{\sim}_c$ the {\it Reeb space} of $c$. Let $V \subset W_c$ be the set of all points representing connected components of preimages of $c$ which contain some singular points of $c$. 

If $W_c$ is regarded as the graph whose vertex set is $V$, then the graph $W_c$ is called the {\it Reeb graph} of $c$.
\end{Def}

In the definition, $q_c:X \rightarrow W_c$ denotes the natural quotient map and $\bar{c}$ denotes the uniquey defined map satisfying the relation $c=\bar{c} \circ q_c$.
For a connected graph without loops which is not a single point or empty, a {\it good function} is a continuous real-valued function on the graph such that the restriction to each edge is injective.

We can give an orientation to each edge compatible with the good function so that the following hold. The {\it starting point} of an edge is a vertex from which the edge originates. The {\it ending point} of an edge is a vertex in which the edge terminates.

\begin{enumerate}
\item The starting point of an edge is the vertex at which the (restriction of the) given good function has the minimum.
\item The ending point of an edge is the vertex at which the (restriction of the) given good function has the maximum.
\end{enumerate}   

The {\it singular set} of a smooth map is defined as the set of all singular points of the map. A {\it singular value} of the map is a point in the manifold of the target such that the preimage contains a singular point. The {\it singular value set} of the map is the image of the singular set. The {\it regular value set} of the map is the complementary set of the singular value set of the map and a regular value of the map, defined before, can be also defined as a point in the regular value set.

In the present paper, {\it Morse} functions and their generalizations are fundamental tools. We omit precise expositions on Morse functions and ({\it $k$-}){\it handles} corresponding to singular points where $k$ is a non-negative integer being smaller than or equal to the dimension of the manifold we consider. We also note here that the singular sets of Morse functions are discrete. For related theory, see \cite{milnor} for example.

A {\it Morse-Bott} function is a function at each singular point which is represented as the composition of a submersion with a Morse function (\cite{bott}). A {\it fold} map is a smooth map such that at each singular point, the map is represented as the product map of a Morse function and the identity map on an open ball: the singular set is a smooth closed submanifold of the domain with no boundary and the restriction to the singular set of the original map is a smooth immersion whose codimension is $1$.

More rigorously, a smooth map $f$ is represented as another smooth map $g$ means that $f$ is {\it $C^{\infty}$ equivalent} to $g$: there exists a pair of a diffeomorphism $\Phi$ between the manifolds of the domains and a diffeomorphism $\phi$ between the targets satisfying the relation $\phi \circ f=g \circ \Phi$.

For related theory and more general theory on singularities of smooth maps, see also \cite{golubitskyguillemin} for example. \cite{saeki1993} is for the class of {\it special generic} maps, which are fold maps regarded as higher dimensional versions of Morse functions with exactly two singular points on spheres. Special generic maps are also used in section \ref{sec:3}.

The following theorem is the main theorem of \cite{kitazawa2}.

\begin{Thm}[\cite{kitazawa2}]
\label{thm:1}
Let $G$ be a connected graph having at least one edge. Assume that a good function $h:G \rightarrow \mathbb{R}$ is given. Assume also that a non-negative integer is assigned to each edge by a map $l$. Then we have a smooth function $f$ of a suitable $3$-dimensional closed, connected and orientable manifold $M$ satisfying the following properties.
\begin{enumerate}
\item The Reeb graph $W_f$ is isomorphic to the given graph $G$ and an isomorphism between the graphs $\phi:W_f \rightarrow G$ exists.
\item If we consider the natural quotient map $q_f$ onto the Reeb graph $W_f$ and a point $p \in W_f$ such that $\phi(p)$ is in the interior of an edge $e$ of $G$ satisfying $l(e)=q \geq 0$, then the preimage ${q_f}^{-1}(p)$ is a closed, connected and orientable surface of genus $q \geq 0$.
\item For a point $p \in M$ mapped by the quotient map $q_f$ to a vertex $v$, $f(p)=h(\phi(v))$.
\item At each singular point of $f$, the function is a Morse function, a Morse-Bott function or represented as the composition of two fold maps.
\end{enumerate}
\end{Thm}

A main theorem is the following.

\begin{Thm}[A main theorem]
\label{thm:2}
Let $G$ be a connected graph having at least one edge. Assume that a good function $h:G \rightarrow \mathbb{R}$ is given. Assume also that an integer $0$ or $1$ is assigned to each edge by a map $l$ satisfying the following conditions.
\begin{itemize}
\item For a vertex at which the good function $h$ does not have a local extremum, the number of edges at which the values of $l$ are $1$ containing the vertex as the starting points and that of edges at which the values of $l$ are $1$ containing the vertex as the ending points agree. Moreover, if the number of edges at which the values are $1$ having the vertex as the starting points is $1$, then there exists at least one edge at which the value of $l$ is $0$ containing the vertex as the starting point or ending point.
\item For a vertex at which the good function $h$ has a local extremum, the number of edges at which the values of $l$ are $1$ containing the vertex is even. 
\end{itemize}
Then, there exist a connected and orientable surface $M$ and a smooth function $f:M \rightarrow \mathbb{R}$ satisfying the following properties.
\begin{enumerate}    
\item The Reeb graph $W_f$ is isomorphic to $G$ and an isomorphism $\phi:W_f \rightarrow G$ exists.
\item If we consider the natural quotient map onto the Reeb graph $W_f$ and a point $p \in W_f$ such that $\phi(p)$ is in the interior of an edge $e$ of $G$ satisfying $l(e)=q$, then the preimage is diffeomorphic to a circle {\rm (}$q=0${\rm )} or a line {\rm (}$q=1${\rm )}.
\item For a point $p \in M$ mapped by the quotient map $q_f$ to a vertex $v$, $f(p)=h(\phi(v))$.   
\item Around each singular point at which the function $f$ does not have a local extremum, the local function is a Morse function.
\item Around each singular at which $f$ has a local extremum, the local function is a Morse function, a Morse-Bott function, or represented as the composition of two Morse functions.
\end{enumerate}
\end{Thm}

This is for smooth functions on closed or open surfaces and similar to Theorem \ref{thm:1} in considerable parts. One of the new ingredients is that the resulting function may have open preimages of regular values. We prove this and a higher dimensional version (Theorem \ref{thm:3}) in the next section. In the last two sections we prove further results related to these theorems.   


\section{Proofs of Theorems \ref{thm:2} and \ref{thm:3}.}
\label{sec:2}
The outline of the proof is similar to that of Theorem \ref{thm:1}. We prove this.
\begin{proof}[Proof of Theorem \ref{thm:2}]
\noindent STEP 1 Construction around a vertex at which the good function $h$ does not have a local extremum. \\
Let $a$ be the number of edges at which the values of $l$ are $1$ containing the vertex as the starting (ending) points (these numbers agree by the assumption on the numbers assigned to edges). Let $b$ and $c$ be the numbers of edges at which the values of $l$ are $0$ containing the vertex as the ending points and the starting points, respectively.
We construct a desired function as a Morse function. We apply well-known fundamental correspondence of the handles in a family of handles attached to the boundary of a manifold and singular points of a Morse function on the manifold produced by the handles. We abuse several terminologies omitting the definitions such as the {\it index} $k$ of a handle or a {\it $k$-handle}. Expositions on related theory are in \cite{milnor} for example.

Let $a$, $b$ and $c$ defined before be arbitrary non-negative integers satisfying $$(a,b,c) \neq (1,0,0),(0,1,0),(0,0,1),(0,0,0).$$. 

This accounts for all cases for the present step.

FIGURE \ref{fig:1} shows handle attachments producing a surface whose boundary is the disjoint union of two $1$-dimensional manifolds which may not be connected. We attach $a-1+b+c$ $1$-handles to the product of a $1$-dimensional manifold $F$, consisting of $a$ copies of a line and $b$ copies of a circle, and the closed interval $[0,1]$. More precisely, we attach them to $F \times \{0\} \subset F \times [0,1]$. If $a>0$, then we attach handles as the following explanation.
\begin{itemize}
\item We attach $a-1$ $1$-handles to connect $a$ copies of a line.
\item We attach $b-1$ $1$-handles to connect $b$ copies of a circle if $b>0$. We attach a $1$-handle to a line to connect the $b$ copies of a circle.
\item We attach $c$ $1$-handles to a line. We may attach $c$ $1$-handles to a circle if it exists, not a line.
\end{itemize}
If $a=0$ and $(b,c) \neq (1,1)$, then we attach $1$-handles according to the following explanation. 
\begin{itemize}
\item We attach $b-1$ $1$-handles to connect $b$ copies of a circle ($b>0$).
\item We attach $c-1$ $1$-handles to one of the $b$ circles ($c>0$).
\end{itemize}
If $a=0$ and $(b,c)=(1,1)$, then we attach $1$-handles according to the following explanation.
\begin{itemize}
\item We choose four disjoint closed intervals in the circle $F$\ ($F \times \{0\}$).
\item We choose two of the four closed intervals which are not adjacent.
\item We attach a $1$-handle to (the disjoint union of) the two closed intervals.
\item We attach another $1$-handle to (the disjoint union of) the remaining two closed intervals.
\end{itemize} 

As a result, we obtain a $2$-dimensional connected and orientable manifold whose boundary is the disjoint union of two $1$-dimensional manifolds. One of the $1$-dimensional manifolds is the $1$-dimensional manifold $F \times \{1\}$, consisting of $a$ lines and $b$ circles. The other is the disjoint union of $a$ lines and $c$ circles.

We obtain a desired local Morse function satisfying the first four properties around the vertex for any $(a,b,c)$ and having exactly one singular point, whose preimage is connected. For the first property, this means the following two.
\begin{itemize}
\item A small regular neighborhood of the vertex in the given graph and the Reeb graph is PL homeomorphic.
\item There exists a PL homeomorphism mapping a vertex onto the point whose preimage (for the quotient map) contains some singular points.
\end{itemize}
In the present paper, situations like this appear in the remaining theorems and proofs and hereafter we omit precise expositions on this. Note that handles can be attached simultaneously to disjoint circles or lines disjointly and that this gives a Morse function on the $2$-dimensional manifold having exactly one singular value in the image. The preimage of the singular value is a connected 1-dimensional polyhedron obtained by identifying finitely many pairs of points of the disjoint union of the $a$ copies of a line and the $b+c$ copies of a circle.

\begin{figure}
\includegraphics[width=30mm]{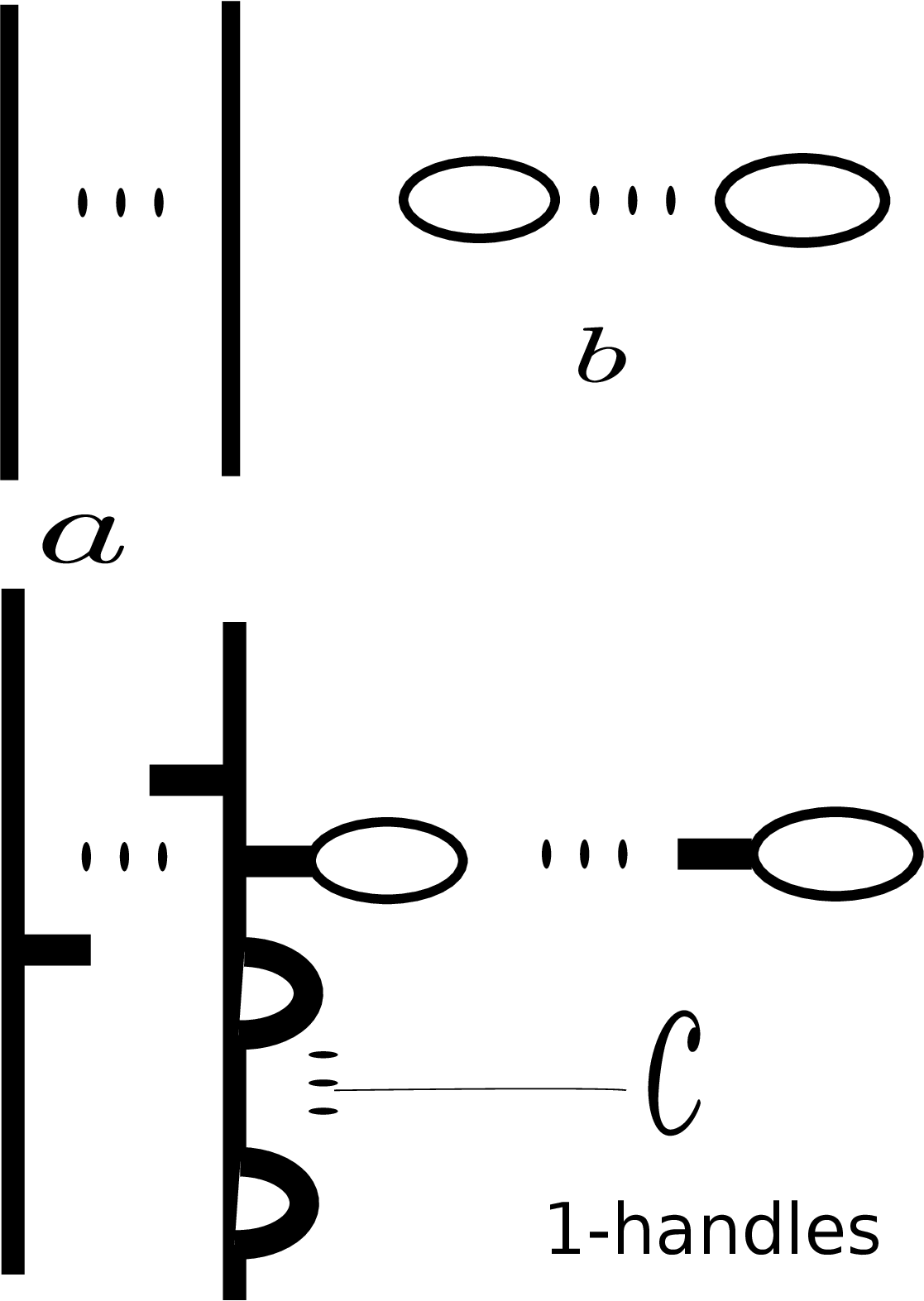}
\caption{Handle attachments to obtain a connected and orientable surface and a desired local Morse function on this in STEP 1 for $a>0$: attaching $a-1$ $1$-handles to connect $a$ copies of a line, $b$ $1$-handles to connect $b$ copies of a circle and a line and $c$ $1$-handles to a line.}
\label{fig:1}
\end{figure} 
\ \\
\noindent STEP 2 Construction around a vertex at which the good function $h$ has a local extremum. \\
Case 1 The case where the vertex is of degree $1$. \\
It is sufficient to consider the case where for the edge $e$ containing the vertex $l(e)=0$ holds by the assumption. A natural {\it height function} of a $2$-dimensional standard disk, which is a Morse function with exactly one singular point in the interior or in the center, is a desired local function. Note that this is generalized to arbitrary dimensional cases and this is a key in proving Theorem \ref{thm:3} later and in various situations of the present paper. A {\it height function} of the unit disk $D^k$ or an $k$-dimensional standard disk is a {\rm (}Morse{\rm )} function represented by the form $x \in D^k \mapsto ||x||^2+c_0$ for a constant $c_0 \in \mathbb{R}$ and natural coordinates. \\
\ \\
\noindent Case 2 The case where the vertex is of degree greater than $1$. \\
We show the case where at the vertex the good function has a local maximum. For the case where at the vertex the good function has a local minimum, we can do similarly.
We first construct a local function as in STEP 1 and embed the image into the plane as the set $\{(x,y) \in {\mathbb{R}}^2 \mid -1 \leq x \leq 1,y=x^2\}$ sending the singular value to the origin. See FIGURE \ref{fig:2}: $F_1$ and $F_2$ represent the preimages of $(-t,t^2)$ and $(t,t^2)$ ($t>0$), respectively. $F_1 \sqcup F_2$ is the preimage of a regular value of the resulting local function, which is obtained by composing the canonical projection. We give a precise explanation.

We construct a desired local function as in STEP 1 first so that $F_1$ is diffeomorphic to the disjoint union of $a$ lines and $b$ circles and that $F_2$ is diffeomorphic to the disjoint union of $a$ lines and $c$ circles where $a$, $b$ and $c$ are arbitrary non-negative integers satisfying $(a,b,c) \neq (1,0,0),(0,1,0),(0,0,1),(0,0,0)$: we consider a case where the vertex is contained in exactly $2a$ edges at which the values of $l$ are $1$ and also contained in exactly $b+c$ edges at which the values of $l$ are $0$.

\begin{figure}
\includegraphics[width=35mm]{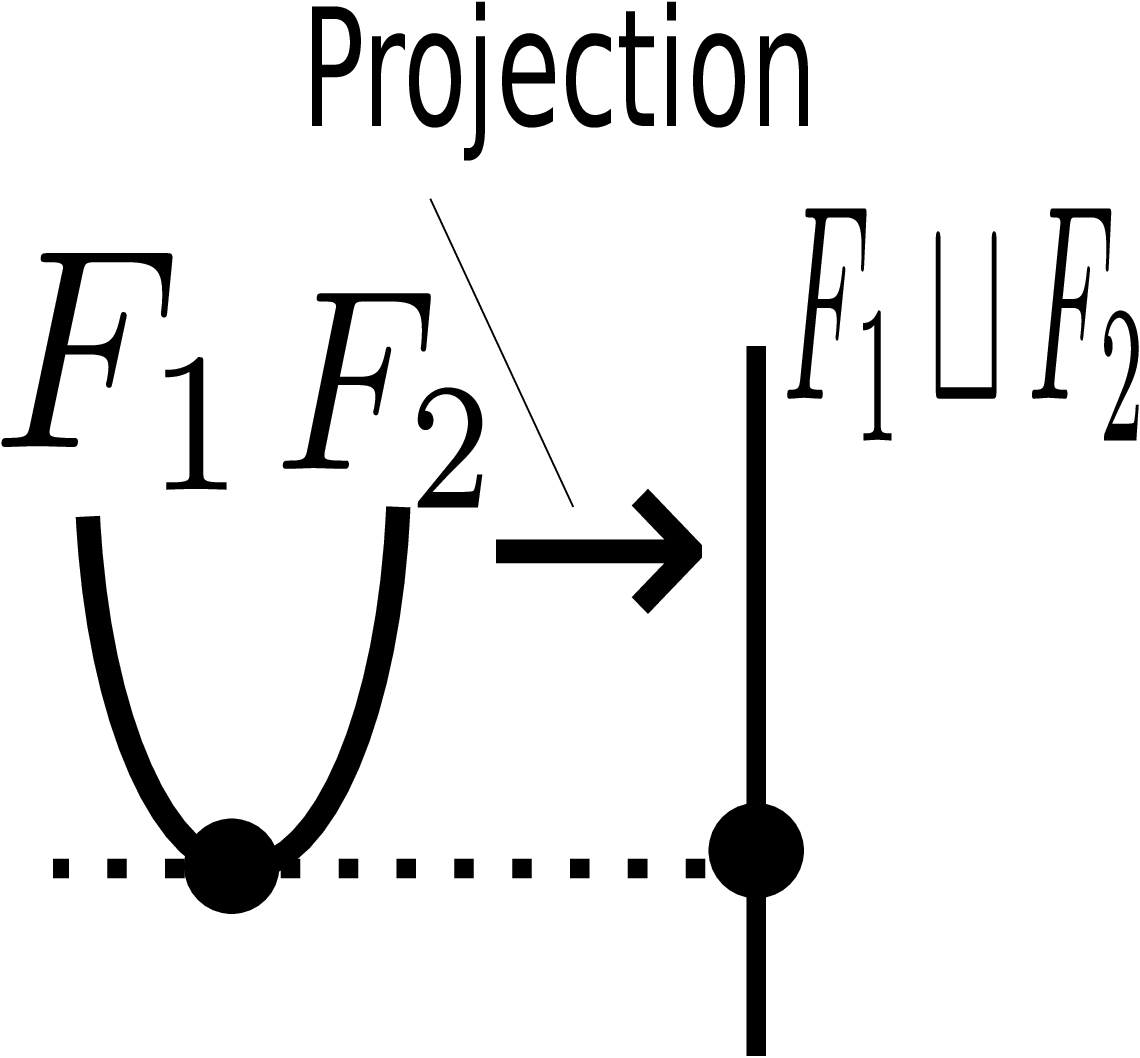}
\caption{Construction of a local function in Case 2.}
\label{fig:2}
\end{figure}

We compose the map into the plane with the projection $p(x,y):=y$ and we have a desired local smooth function.
By the construction, at each singular point, the function is a Morse-Bott function or represented as the composition of two Morse functions. As a result, we obtain a local function satisfying the first, second, third and fifth properties around the vertex.

In the case $(a,b,c)=(1,0,0),(0,1,1)$, we construct a trivial smooth bundle whose fiber is connected and diffeomorphic to $S^1$ or $\mathbb{R}$ instead of a Morse function before embedding the image into the plane. In this case, at the singular points, the function is Morse-Bott. 

This completes the proof for Case 2 in the present step for any $(a,b,c)$ and any case. \\
\ \\
\noindent STEP 3 Completing the construction. \\
Last, we construct functions around remaining parts. We can construct these functions as trivial smooth bundles. Gluing the local functions together on the $1$-dimensional manifolds in boundaries gives a desired function $f$ on a surface $M$. To make the resulting surface $M$ orientable, we must use the diffeomorphisms for the gluing carefully one after another and we can do. We also have the Reeb space of the local function and a PL homeomorphism from the Reeb space to a small regular neighborhood of the vertex in each STEP and each CASE as presented in STEP 1 and we also have a desired isomorphism $\phi:W_f \rightarrow G$. Note also that we must scale each local function so that the value at the singular point $p$ is $h(\phi(q_f(p))$ in each STEP and each CASE.

This completes the proof.
\end{proof}

We can show a higher dimensional version of Theorem \ref{thm:2} or Theorem \ref{thm:3} similarly by virtue of Remark \ref{rem:1} in the following.

\begin{Rem}
\label{rem:1}
Let $n>1$ be an integer. Let $a$, $b$ and $c$ be non-negative integers satisfying $$(a,b,c) \neq (1,0,0),(0,1,0),(0,0,1),(0,0,0).$$
Consider the disjoint union $F$ of $a$ copies of ${\mathbb{R}}^n$ and $b$ copies of $S^n$ and $F \times [0,1]$. We can attach $1$-handles to $F \times \{0\}$ and $n$-handles to this. As a result, for an arbitrary positive integer $c$, we obtain an ($n+1$)-dimensional connected and orientable manifold whose boundary is the disjoint union of two $n$-dimensional manifolds. One is the disjoint union of $a$ copies of ${\mathbb{R}}^n$ and $b$ copies of $S^n$ or $F \times \{1\}$. The other is the disjoint union of $a$ copies of ${\mathbb{R}}^n$ and $c$ copies of $S^n$. 

FIGURE \ref{fig:3} represents the attachments of $1$-handles and $n$-handles to the disjoint union of the manifolds diffeomorphic to ${\mathbb{R}}^n$ in $F \times \{0\}$ for $a>0$: we attach $a-1$ $1$-handles to connect $a$ copies of ${\mathbb{R}}^n$, an $n$-handle to each copy of ${\mathbb{R}}^n$ in the figure. We also need to attach $b$ $1$-handles to connect $b$ copies of $S^n$ and a copy of ${\mathbb{R}}^n$ in $F \times \{0\}$. In addition, we also need to attach $c$ $n$-handles to a copy of ${\mathbb{R}}^n$ in $F \times \{0\}$. If $a=0$ and $(b,c) \neq (1,1)$, then we attach $b$ $1$-handles to connect $b$ copies of $S^n$ and $c-1$ $n$-handles to a copy of $S^n$ in $F \times \{0\}$. If $a=0$ and $(b,c) = (1,1)$, then we attach a $1$-handle and an $n$-handle to the standard $n$-dimensional sphere in $F \times \{0\}$ to obtain a desired ($n+1$)-dimensional manifold.

Note that the handles are attached simultaneously in these cases. This yields STEP 1 of the proof of Theorem \ref{thm:2} in the proof of Theorem \ref{thm:3}. We can prove STEP 2 similarly. Here, in Case 1, we use a height function of an $n$-dimensional standard disk instead.
\end{Rem}
\begin{figure}
\includegraphics[width=30mm]{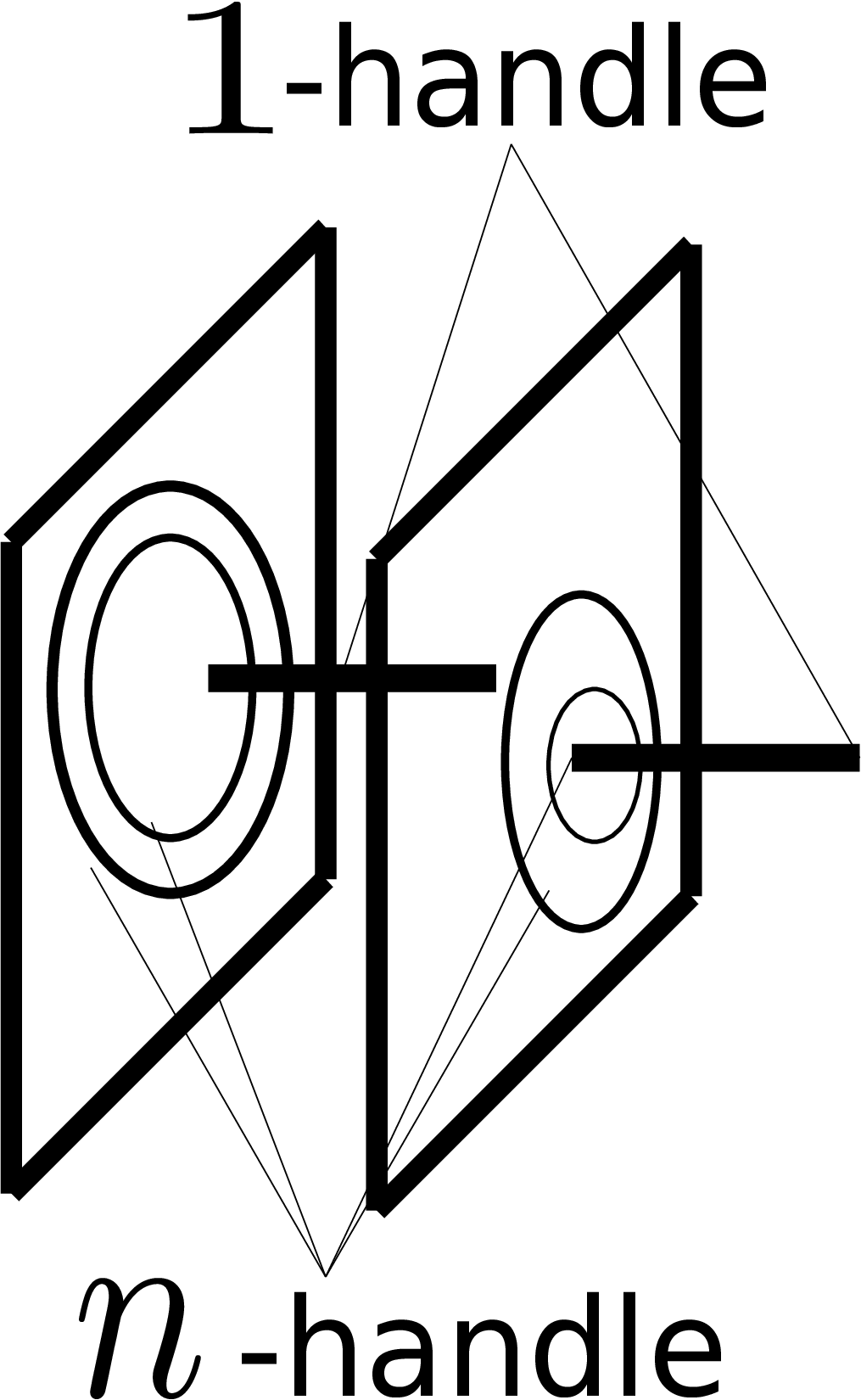}
\caption{A higher dimensional version of FIGURE \ref{fig:1} for $a>0$. We attach the $a-1$ $1$-handles and the $a$ $n$-handles to the disjoint union of manifolds diffeomorphic to ${\mathbb{R}}^n$ : for $a$ $n$-handles, only submanifolds diffeomorphic to $S^{n-1} \times [-1,1]=S^{n-1} \times D^1$ to which the $n$-handles are attached are depicted. Other handles and the attachments to copies of $S^n$ are omitted.}
\label{fig:3}
\end{figure}

We finish this section by presenting the higher dimensional version or Theorem \ref{thm:3}. Rigorous proofs are left to readers.

\begin{Thm}
\label{thm:3}
Let $n$ be a positive integer.
Let $G$ be a connected graph having at least one edge. Assume that a good function $h:G \rightarrow \mathbb{R}$ is given. Assume also that an integer $0$ or $1$ is assigned to each edge by a map $l$ satisfying the following conditions.
\begin{itemize}
\item For a vertex at which the good function does not have a local extremum, the number of edges at which the values of $l$ are $1$ containing the vertex as the starting points and that of edges at which the values of $l$ are $1$ containing the vertex as the ending points agree. Moreover, if the number of edges at which the values of $l$ are $1$ containing the vertex as the starting points is $1$, then there exists an edge at which the value of $l$ is $0$ containing the vertex as the starting point or ending point.
\item For a vertex at which the good function has a local extremum, the number of edges at which the values of $l$ are $1$ are even. 
\end{itemize}
Then there exist an {\rm (}$n+1${\rm )}-dimensional connected and orientable manifold and a smooth function $f:M \rightarrow \mathbb{R}$ satisfying the following properties.
\begin{enumerate}
\item The Reeb graph $W_f$ is isomorphic to $G$ and an isomorphism $\phi:W_f \rightarrow G$ exists.
\item If we consider the natural quotient map onto the Reeb graph $W_f$ and a point $p \in W_f$ such that $\phi(p)$ is in the interior of an edge $e$ of $G$ satisfying $l(e)=q$, then the preimage is diffeomorphic to $S^n$ {\rm (}$q=0${\rm )} or ${\mathbb{R}}^n$ {\rm (}$q=1${\rm )}. 
\item For a point $p \in M$ mapped by the quotient map $q_f$ to a vertex $v$, $f(p)=h(\phi(v))$.   
\item At each singular point at which the function $f$ does not have a local extremum, the local function is a Morse function.
\item At each singular point at which the function $f$ has a local extremum, the local function is a Morse function, Morse-Bott function, or represented as the composition of two Morse functions.
\end{enumerate}
\end{Thm}

\section{Modifications of Theorems \ref{thm:2} and \ref{thm:3} where the original assumptions do not hold.}
\label{sec:3}
\begin{Thm}
\label{thm:4}
Let $G$ be a connected graph having at least one edge. Assume that a good function $h:G \rightarrow \mathbb{R}$ is given. Assume also that an integer $0$ or $1$ is assigned to each edge by a map $l$. We assume that the original assumption of Theorem \ref{thm:2} does not hold.

Then there exist a connected an orientable surface $M$ and a smooth function $f:M \rightarrow \mathbb{R}$ satisfying the following properties.
\begin{enumerate}
\item The Reeb graph $W_f$ is isomorphic to $G$ and an isomorphism $\phi:W_f \rightarrow G$ exists.
\item If we consider the natural quotient map onto the graph and for each point that is not a vertex and that is in an edge $e$ satisfying $l(e)=q$, then the preimage is diffeomorphic to a circle {\rm (}$q=0${\rm )} or a line {\rm (}$q=1${\rm )}.
\item For a point $p \in M$ mapped by the quotient map $q_f$ to a vertex $v$, $f(p)=h(\phi(v))$. 
\item At singular points at which the function $f$ does not have local extrema, except finitely many ones, the local functions are Morse functions or Morse-Bott functions.
\item At singular points at which the function $f$ have local extrema, except finitely many ones, the local functions are Morse functions, Morse-Bott functions, or represented as the compositions of two Morse functions or the compositions of Morse-Bott functions with Morse functions.
\end{enumerate}
\end{Thm}
\begin{proof}
The proof is similar to that of Theorem \ref{thm:2} in considerable parts. However, there are several new ingredients.\\
 \\
\noindent STEP 1 Construction around a vertex at which the good function $h$ does not have a local extremum. \\
Let $a$ and $d$ be the number of edges at which the values of $l$ are $1$ containing the vertex as the starting points and the ending points, respectively. Let $b$ and $c$ be the numbers of edges at which the values or $l$ are $0$ containing the vertex as the ending points and the starting points, respectively.

By the assumption, the assumption of Theorem \ref{thm:2} does not hold. This means that either of the following holds.
\begin{itemize}
\item $a \neq d$ and either of the following hold.
\begin{itemize}
\item $a \neq 0$ and $d \neq 0$.
\item $a=0$ and $b>0$.
\item $d=0$ and $c>0$.
\end{itemize} 
\item $a=d=1$ and $(b,c)=(0,0)$.
\end{itemize}
\ \\
\noindent CASE 1-A $a=d=1$ and $(b,c)=(0,0)$ hold.
Let $p_0>0$. Let $g_0$ be a smooth function on $\mathbb{R}$ such that $g_0(x)=0$ for $x \geq p_0$ and $g_0(x_1)-g_0(x_2)>0$ for any $x_1<x_2 \leq p_0$. We introduce several points and subsets in the plane. Set $A_0:=\{(t,0) \mid t>p_0\}$ and $B_0:=\{(t,g_0(t)) \mid p_0-1<t<p_0\}$. We also set $-A_0:=\{-a \mid a \in A_0\}$ and $-B_0:=\{-b \mid b \in B_0\}$.

We can construct a smooth map $F_0$ from a surface into the plane satisfying the following properties.
\begin{itemize}
\item $F_0$ is a fold map from an open surface into the plane the interior of whose image is the domain bounded by the union of the following sets defined uniquely by the following condition: the image of the composition of $F_0$ with the projection $p(x,y):=y$ is $(-g_0(p_0-1),g_0(p_0-1)) \subset \mathbb{R}$.
\begin{itemize}
\item $A_0$.
\item $\{(p_0,0)\}$.
\item $B_0$.
\item $-A_0$.
\item $\{(-p_0,0)\}$.
\item $-B_0$.
\item $\{(t,g_0(p_0-1)) \mid t \leq p_0-1\}$. 
\item $\{(t,-g_0(p_0-1)) \mid t \geq 1-p_0\}$. 
\end{itemize}
\item Let $D_0$ denote the domain before. The image of $F_0$ is the disjoint union of $D_0$ and the first six sets before. Moreover, the union of the first six set is the singular value set of the fold map $F_0$ and the restriction of $F_0$ to the singular set is an embedding.
\item
Let ${\epsilon}_0>0$ be a small number.
$L_{0,1}:=D_0 \bigcap \{(t,g_0(t)-{\epsilon}_0) \mid t \in \mathbb{R}\}$. Over the union of the domain $U_1$ between the union $A_0 \bigcup \{(p_0,0)\} \bigcup B_0 \subset D_0$ and $L_{0,1} \subset D_0$ and the union $A_0 \bigcup \{(p_0,0)\} \bigcup B_0$, $F_0$ is represented as the product of a Morse function on an interval with exactly one singular point in the interior and the identity map on the union $A_0 \bigcup \{(p_0,0)\} \bigcup B_0$, which is a line: note that the graph of the Morse function on an interval is a parabola as in FIGURE \ref{fig:2}.  
\item
$L_{0,2}:=D_0 \bigcap \{(t,-g_0(t)+{\epsilon}_0) \mid t \in \mathbb{R}\}$. Over the union of the domain $U_2$ between the union $-A_0 \bigcup \{(-p_0,0)\} \bigcup -B_0 \subset D_0$ and $L_{0,2} \subset D_0$ and the union $-A_0 \bigcup \{(-p_0,0)\} \bigcup -B_0$, $F_0$ is represented as the product of a Morse function on a closed interval with exactly one singular point in the interior and the identity map on the union $-A_0 \bigcup \{(-p_0,0)\} \bigcup -B_0$, which is a line: note that the graph of the Morse function on an interval is a parabola as in FIGURE \ref{fig:2}.  

\item Over the complement of the union of $U_1 \sqcup U_2$ of the two domains in the previous two properties in $D_0$, $F_0$ is represented as the projection of a trivial smooth bundle whose fiber is a two-point set. 
\item Canonically, we can extend $F_0$ to a smooth map $\tilde{F_0}$ from a surface which is not closed or open. More precisely, this satisfies the following conditions.
\begin{itemize}
\item The image of $\tilde{F_0}$ is the closure of $D_0$.
\item The preimages of $(p_0-1,g_0(p_0-1))$ and $(1-p_0,g_0(p_0-1))$ are single points.
\item The preimages of points in the complement of the union of the image of $F_0$ and $\{(p_0-1,g_0(p_0-1)),(1-p_0,g_0(p_0-1))\}$ in the image of $\tilde{F_0}$ consist of two points. 
\end{itemize}
\end{itemize}

We can define the composition of $\tilde{F_0}$ with the projection $p(x,y):=y$. The singular value set is $\{0\}$. See also FIGURE \ref{fig:4_0}. After scaling, we can obtain a desired local function satisfying the first, second, third and fourth conditions in the statement. By the construction, at each singular point, except finitely many ones, the local function is a Morse-Bott function. More precisely, the preimage of points in $A_0$ and $-A_0$ by $\tilde{F_0}$ are single points and singular points at which the functions are Morse-Bott functions.

\begin{figure}
\includegraphics[width=35mm]{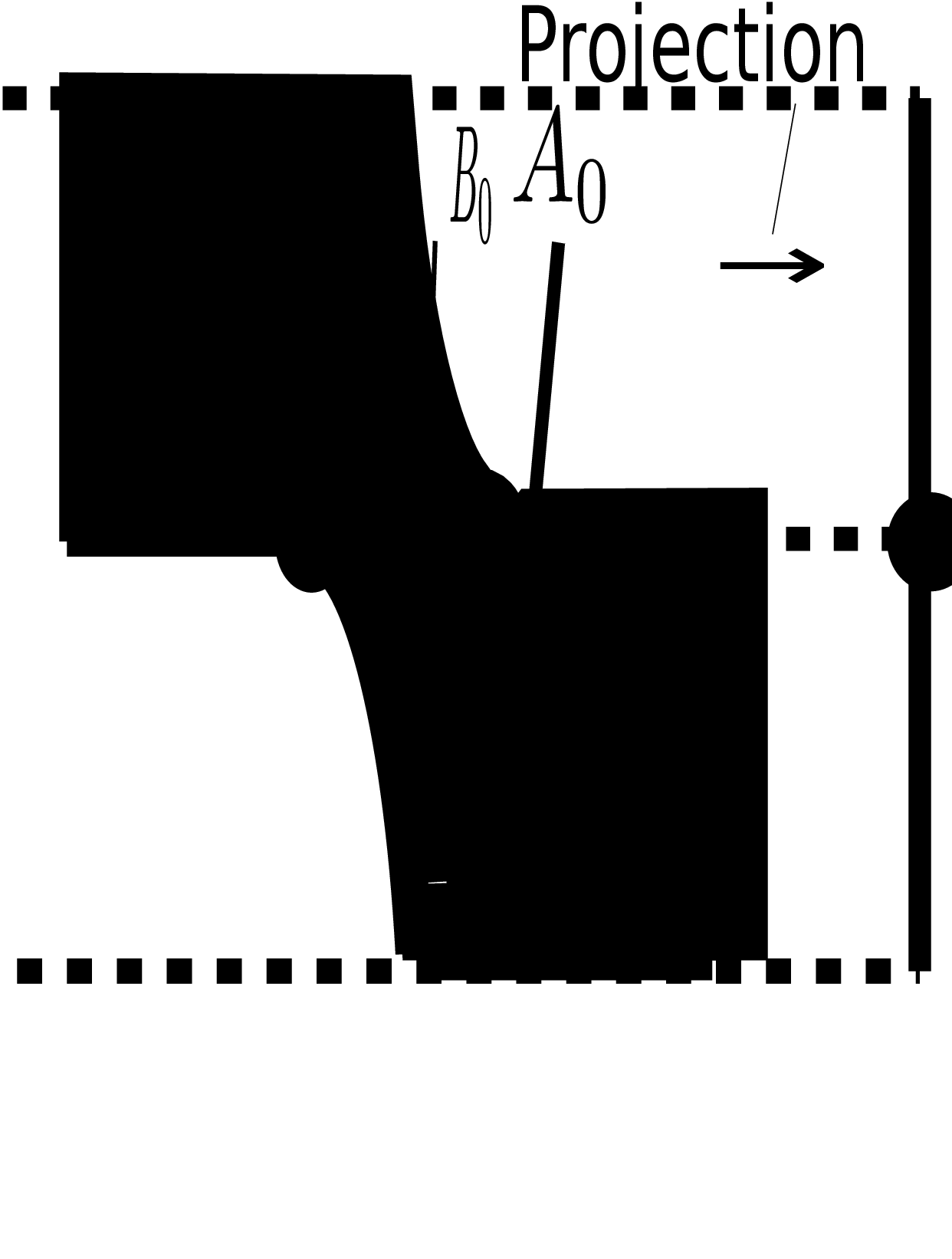}
\caption{The image of a smooth map $\tilde{F}$ into the plane and the composition of the map with the projection $p$: the dot on the line in the right is the singular value.}
\label{fig:4_0}
\end{figure}
\ \\
\noindent CASE 1-B Either of the following holds.
\begin{itemize}
\item  $a \neq d$, $a \neq 0$ and $d \neq 0$ hold.
\item $a \neq d$, $b>0$, $c>0$ hold and either $a=0$ or $d=0$ holds.
\end{itemize}

Suppose $d>a$.

First, we construct a local Morse function as in Theorem \ref{thm:2} where "$a$", "$b$" and "$c$" in the proof of Theorem \ref{thm:2} are $a$, $b$ and $c$ here respectively except the case $(a,b,c)=(1,0,0),(0,1,1)$, a local smooth function in CASE 1-A for the case $(a,b,c)=(1,0,0)$, and a trivial smooth $S^1$-bundle over a closed or open interval for the case $(a,b,c)=(0,1,1)$. By fundamental properties on structures of these functions, we can construct a local function so that we can find disjoint $d-a$ copies of small smooth trivial bundles whose fibers are intervals over the image and that the projections of the bundles are regarded as the restrictions of the original function here. To obtain a desired function, we change these $d-a$ functions to new smooth functions. We construct the function.

Let $g$ be a smooth function on $\mathbb{R}$ such that $g(x)=0$ for $x \leq 0$ and $g(x_2)-g(x_1)>0$ for any $0 \leq x_1<x_2$ (set $g(x)=e^{-\frac{1}{x}}$ for $x>0$ for example). Let $p_1$ and $p_2$ be positive numbers. Let $p_3<p_4$ be positive numbers satisfying $p_3>p_2$.
We introduce several points and subsets in the plane. Set $A:=\{(t,0) \mid t<0\}$, $B:=\{(t,-g(t)) \mid 0<t<p_1\}$, $C:=\{(t,(t-p_2)^2) \mid 0 \leq (t-p_2)^2<{(p_3-p_2)}^2\}$ and $O:=(0,0)$.
We can construct a smooth map $F$ from a surface into the plane satisfying the following properties.
\begin{itemize}
\item $F$ is a fold map from an open surface into the plane the interior of whose image is the domain bounded by the union of the following sets defined uniquely by the following condition: the image of the composition of $F$ with the projection $p(x,y):=y$ is $(-g(p_1),(p_3-p_2)^2) \subset \mathbb{R}$.
\begin{itemize}
\item $A$.
\item $\{O\}$.
\item $B$.
\item $\{(t,-g(p_1))\mid p_1 \leq t \leq p_4\}$.
\item $\{(p_4,t) \mid -g(p_1) \leq t \leq (p_3-p_2)^2\}$.
\item $\{(t,{(p_3-p_2)}^2)\mid p_3 \leq t \leq p_4\}$.
\item $C$.
\item $\{(t,{(p_3-p_2)}^2)\mid t \leq p_2-(p_3-p_2)\}$.
\end{itemize}
\item Let $D$ denote the domain before. The image of $F$ is the disjoint union of $D$, $A$, $\{O\}$, $B$ and $C$. Moreover, the union $A \bigcup \{O\} \bigcup B \bigcup C$ is the singular value set of the fold map $F$ and the restriction of $F$ to the singular set is an embedding.
\item
Let $\epsilon>0$ be a small number.
$L_1:=D \bigcap \{(t,-g(t)+\epsilon) \mid t \in \mathbb{R}\}$. Over the union of the domain $U_3$ between the union of $A \bigcup \{O\} \bigcup B \subset D$ and $L_1 \subset D$ and the union of $A \bigcup \{O\} \bigcup B$, $F$ is represented as the product of a Morse function on an interval with exactly one singular point and the identity map on the union
 of $A \bigcup \{O\} \bigcup B$, which is a line: note that the graph of the Morse function on an interval is a parabola as in FIGURE \ref{fig:2}.  
\item
Let $L_2:=D \bigcap \{(t,{(t-p_2)}^2-\epsilon) \mid t \in \mathbb{R}\}$. Over the union of the domain $U_4$ between $C \subset D$ and $L_2 \subset D$ and $C$, $F$ is represented as the product of a Morse function on an interval with exactly one singular point as just before and the identity map on $C$, which is a line.  
\item Over the complement of the union $U_3 \sqcup U_4$ of the two domains before in $D$, $F$ is represented as a trivial smooth bundle whose fiber is a disjoint union of two points. 
\item Canonically, we can extend $F$ to a smooth map $\tilde{F}$ on a surface which is not closed or open. More precisely, this satisfies the following properties.
\begin{itemize}
\item The image of $\tilde{F}$ is the closure of $D$.
\item The preimages of $(p_1,-g(p_1))$, $(p_3,{(p_3-p_2)}^2)$ and $(p_2-(p_3-p_2),{(p_3-p_2)}^2)$ are single points.
\item The preimages of points in the complement of the union of the image of $F$ and $\{(p_1,-g(p_1)),(p_3,{(p_3-p_2)}^2),(p_2-(p_3-p_2),{(p_3-p_2)}^2)\}$ in the image of $\tilde{F}$ consist of two points. 
\end{itemize}
\end{itemize}

We can define the composition of $\tilde{F}$ with the projection $p(x,y):=y$. The singular value set is $\{0\}$. See also FIGURE \ref{fig:4}. After scaling, we can obtain desired new functions instead of $d-a$ trivial bundles. Thus we can obtain a desired local function satisfying the first, second, third and fourth conditions in the statement. By the construction, at each singular point, except finitely many ones, the local function is a Morse or a Morse-Bott function. More precisely, the preimage of $(p_2,0)$ by $\tilde{F}$ is a single point and a singular point at which the function is Morse. The preimages of points in $A$ by $\tilde{F}$ are single points and singular points at which the functions are Morse-Bott functions. FIGURE \ref{fig:5} represents a local deformation of preimages when the value of the local function $p \circ \tilde{F}$ increases.

\begin{figure}
\includegraphics[width=35mm]{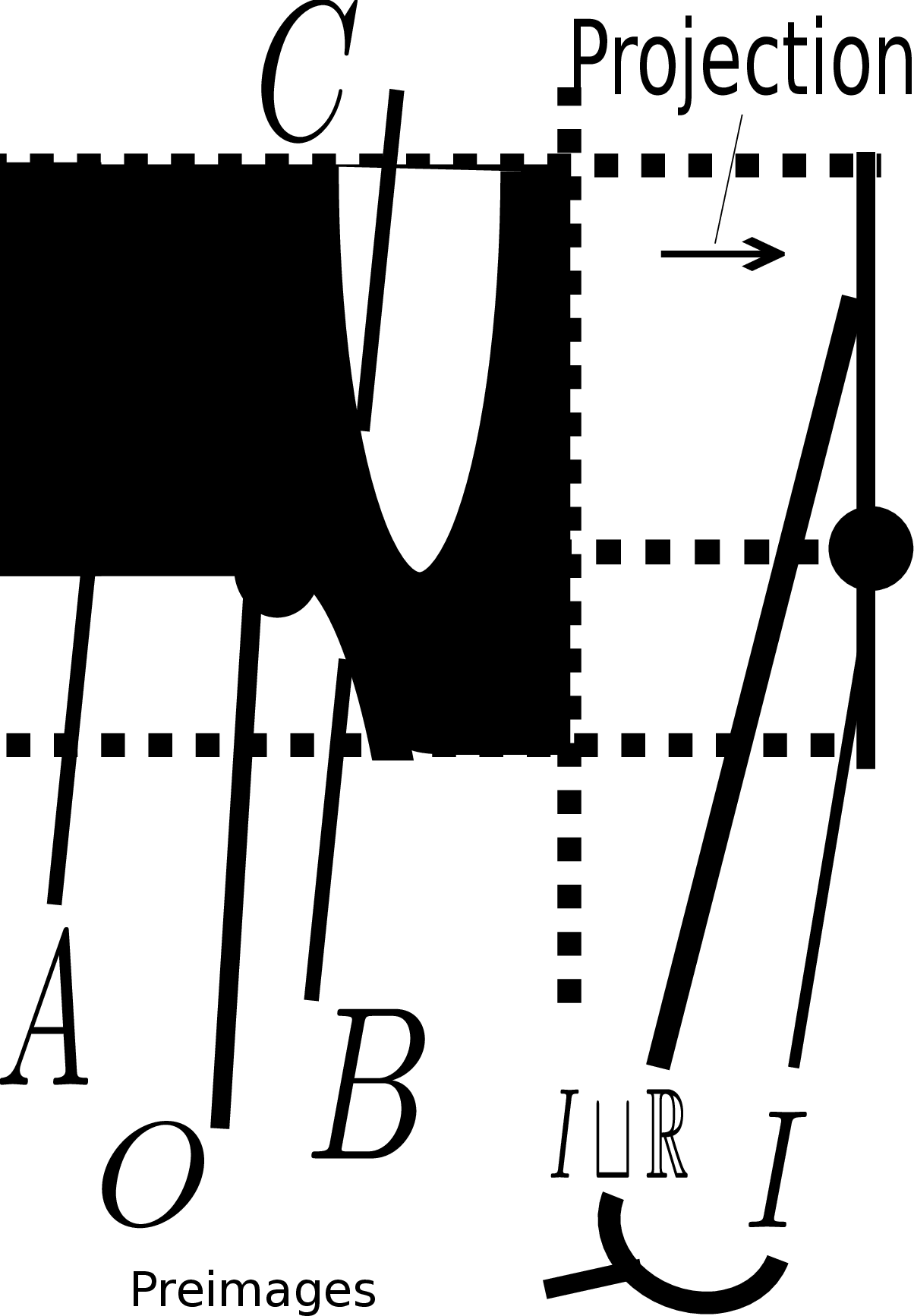}
\caption{The image of a smooth map $\tilde{F}$ into the plane and the composition of the map with the projection $p$: $I$ in the explanation of preimages of regular values denotes a closed interval and the dot on the line in the right is the singular value.}
\label{fig:4}
\end{figure}

\begin{figure}
\includegraphics[width=35mm]{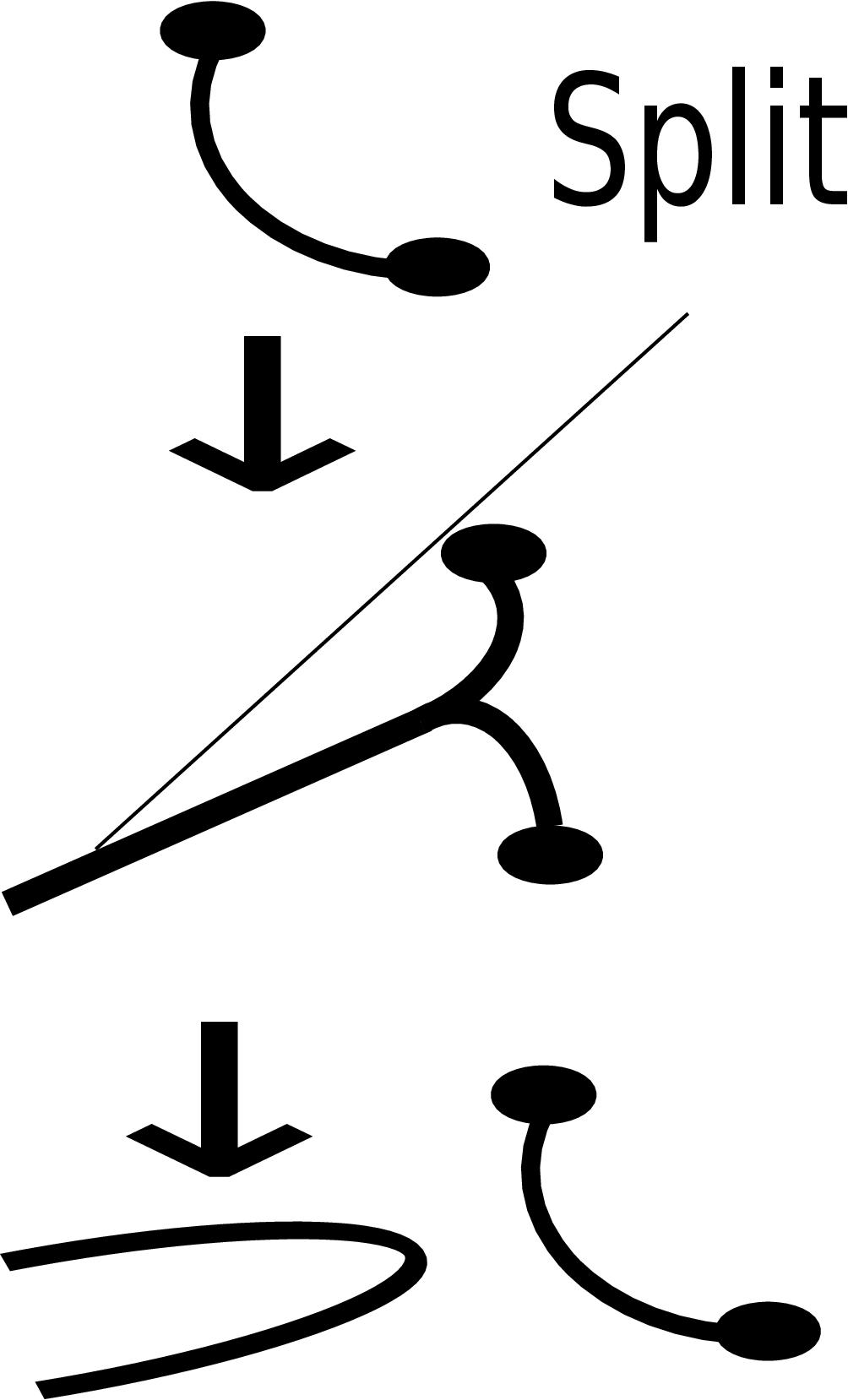}
\caption{A local deformation of preimages when the value of the function increases.}
\label{fig:5}
\end{figure}

In the case $a>d$, we consider the minus of the function obtained above. \\
 \\
\noindent CASE 1-C $a \neq d$, $a=0$, $b>0$, and $c=0$ hold.\\
First, we construct a local Morse function as in Theorem \ref{thm:2} where "$a$" and "$b$" in the proof of Theorem \ref{thm:2} are $a$ and $b$ respectively and "$c$" in the proof is $1$.
As we did in CASE 1-B, we consider $d-a=d$ trivial bundles and change them into new functions. One of these functions is changed into a new function. We construct this new function.

Let $g$ be a smooth function on $\mathbb{R}$ such that $g(x)=0$ for $x \leq 0$ and $g(x_2)-g(x_1)>0$ for any $0 \leq x_1<x_2$ (set $g(x)=e^{-\frac{1}{x}}$ as in CASE 1-B). Let $p_1$ and $p_2$ be positive numbers and $p_3<p_4$ be positive numbers satisfying $p_3>p_2$ as in CASE 1-B.
We introduce several points and subsets in the plane. We set $A^{\prime}:=\{(t,0) \mid t<0\}$, $B^{\prime}:=\{(t,-g(t)) \mid 0<t<p_1\}$ and $O:=(0,0)$ as before.

We can construct a smooth map $F^{\prime}$ from a surface into the plane satisfying the following properties.
\begin{itemize}
\item $F^{\prime}$ is a fold map from an open surface into the plane the interior of whose image is the domain bounded by the union of the following sets defined uniquely by the following condition: the image of the composition of $F^{\prime}$ with the projection $p(x,y):=y$ is $(-g(p_1),(p_3-p_2)^2) \subset \mathbb{R}$.
\begin{itemize}
\item $A^{\prime}$.
\item $\{O\}$.
\item $B^{\prime}$.
\item $\{(t,-g(p_1))\mid p_1 \leq t \leq p_4\}$.
\item $\{(p_4,t) \mid -g(p_1) \leq t \leq (p_3-p_2)^2\}$.
\item $\{(t,{(p_3-p_2)}^2)\mid t \leq p_4\}$.
\end{itemize}
\item Let $D^{\prime}$ denote the domain before. The image of $F^{\prime}$ is the disjoint union of $D^{\prime}$, $A^{\prime}$, $\{O\}$ and $B^{\prime}$. Moreover, the union of $A^{\prime} \bigcup \{O\} \bigcup B^{\prime}$ is the singular value set of the fold map $F^{\prime}$ and the restriction of $F^{\prime}$ to the singular set is an embedding.
\item
Let ${\epsilon}^{\prime}>0$ be a small number.
${L_1}^{\prime}:=D^{\prime} \bigcap \{(t,-g(t)+{\epsilon}^{\prime}) \mid t \in \mathbb{R}\}$. Over the union of the domain $U_5$ between the union of $A^{\prime} \bigcup \{O\} \bigcup B^{\prime} \subset D$ and ${L_1}^{\prime} \subset D$ and the union $A^{\prime} \bigcup \{O\} \bigcup B^{\prime}$, $F^{\prime}$ is represented as the product of a Morse function on an interval with exactly one singular point and the identity map on the union
 of $A^{\prime} \bigcup \{O\} \bigcup B^{\prime}$, which is a line: note that the graph of the Morse function on an interval is a parabola as in FIGURE \ref{fig:2}.  
\item Over the complement of the domain $U_5$ in the previous property in $D^{\prime}$, $F^{\prime}$ is represented as the projection of a trivial smooth bundle whose fiber is a two-point set. 
\item Canonically, we can extend $F^{\prime}$ to a smooth map $\tilde{F^{\prime}}$ on a surface which is not closed or open. More precisely, this satisfies the following properties.
\begin{itemize}
\item The image of $\tilde{F^{\prime}}$ is the closure of $D^{\prime}$.
\item The preimage of $(p_1,-g(p_1))$ is a single point.
\item The preimages of points in the complement of the union of the image of $F^{\prime}$ and $\{(p_1,-g(p_1))\}$ in the image of $\tilde{F^{\prime}}$ consist of two points. 
\end{itemize}
\end{itemize}

We can define the composition of $\tilde{F^{\prime}}$ with the projection $p(x,y):=y$. The singular value set is $\{0\}$. See also FIGURE \ref{fig:6}. 

After scaling, we can obtain the new function. We remove the original $d-a$ trivial bundles as in CASE 1-B. We use this new function and $d-a-1$ copies of a function used in CASE 1-B instead of the trivial bundles. We can obtain a desired local function on an orientable surface satisfying the first, second, third and fourth properties in the statement. By the construction, at each singular point, except finitely many ones, the local function is a Morse-Bott function. The preimages of points in $A$ by $\tilde{F^{\prime}}$ are single points and singular points at which the functions are Morse-Bott functions. FIGURE \ref{fig:7} represents a local deformation of preimages when the value of the local function $p \circ \tilde{F^{\prime}}$ increases.

\begin{figure}
\includegraphics[width=35mm]{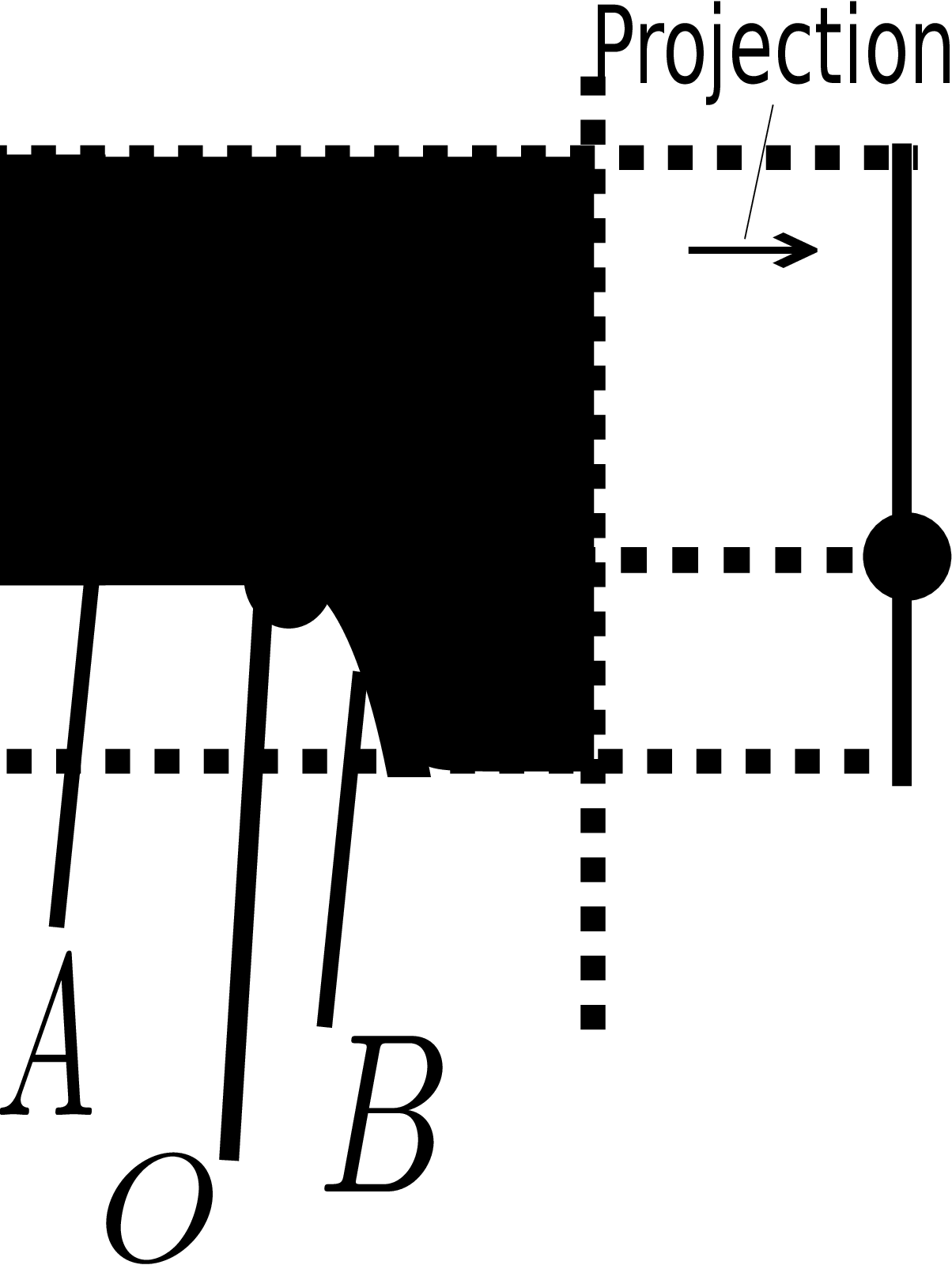}
\caption{The image of a smooth map $\tilde{F^{\prime}}$ into the plane and the composition of the map with the projection $p$: the dot on the line in the right is the singular value.}
\label{fig:6}
\end{figure}

\begin{figure}
\includegraphics[width=35mm]{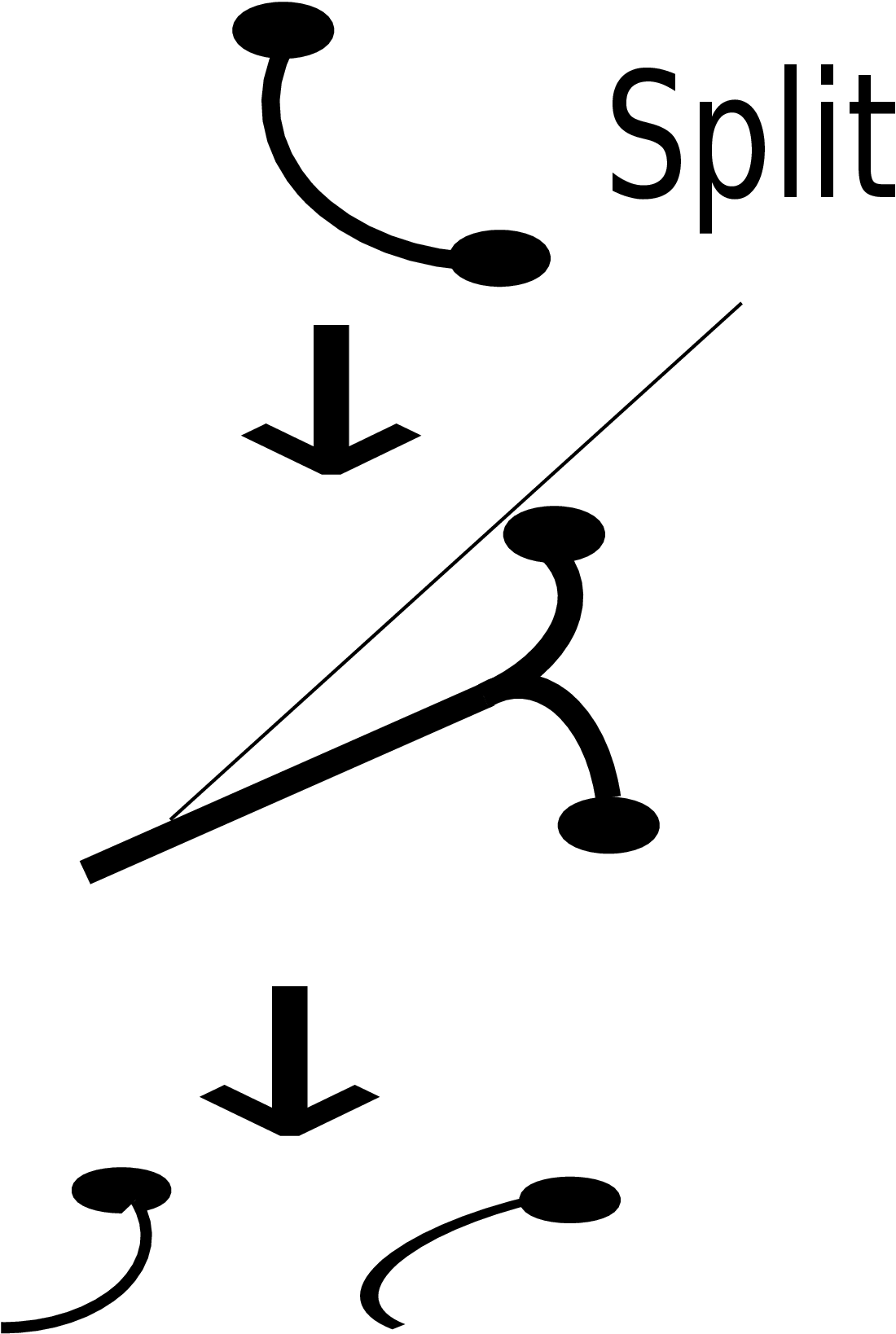}
\caption{A local deformation of preimages when the value of the function increases.}
\label{fig:7}
\end{figure}
\ \\
\noindent CASE 1-D $a \neq d$, $d=0$, $c>0$, and $b=0$ hold.
we consider the minus of the function obtained in CASE 1-C.

We have considered all cases and this completes STEP 1. \\

\noindent STEP 2 Construction around a vertex at which the good function $h$ has a local extremum. \\

If the vertex is of degree larger than $1$, then we can argue as we do in the proof of Theorem \ref{thm:2} using the local function. 

If the vertex is of degree $1$, then $l(e)=1$ for the edge $e$ containing the vertex.

Let $g$ be a smooth function on $\mathbb{R}$ such that $g(x)=0$ for $x \leq 0$ and $g(x_2)-g(x_1)>0$ for any $0 \leq x_1<x_2$ as in STEP 1. We introduce several points and subsets in the plane as in STEP 1. Set $A_1:=\{(t,0) \mid t<0\}$ and $B_1:=\{(t,g(t)) \mid 0<t<p\}$ for a positive number $p>0$. Set $O:=(0,0)$ also in this case.

We can construct a smooth map $F_1$ from a surface into the plane satisfying the following properties.
\begin{itemize}
\item $F_1$ is a fold map from an open surface into the plane the interior of whose image is the domain bounded by the union of the following sets defined uniquely by the following condition: the image of the composition of this with the projection $p(x,y):=y$ is $[0,g(p)) \subset \mathbb{R}$.
\begin{itemize}
\item $A_1$.
\item $\{O\}$.
\item $B_1$.
\item $\{(t,g(p))\mid t \leq p\}$.
\end{itemize}
\item Let $D_1$ denote the domain before. The image of $F_1$ is the union of $D_1$, $A_1$, $\{O\}$ and $B_1$. Moreover, the union $A_1 \bigcup \{O\} \bigcup B_1$ is the singular value set of the fold map $F_1$ and the restriction of $F_1$ to the singular set is an embedding. Moreover, as in the case for the maps $F_0$, $F$, $F^{\prime}$ and $F$ before, over a suitable small collar neighborhood of the singular value set and the preimage, $F_1$ is represented as the product map of a Morse function on an interval with exactly one singular point and the identity map on the union $A_1 \bigcup \{O\} \bigcup B_1$.
\item Canonically, we can extend $F_1$ to a smooth map $\tilde{F_1}$ on a surface which is not closed or open. More precisely, this satisfies the following properties.
\begin{itemize}
\item The image of $\tilde{F_1}$ is the closure of $D_1$.
\item The preimage of $(p,g(p))$ is a single point.
\item The preimages of points in the complement of the union of the image of $F_1$ and $\{(p,g(p))\}$ in the image of $\tilde{F_1}$ consist of two points. 
\end{itemize}
\end{itemize}

We can define the composition of $\tilde{F_1}$ with the projection $p(x,y):=y$. The singular value set is $\{0\}$. See also FIGURE \ref{fig:8}. After scaling, we can obtain a desired local function. 

\begin{figure}
\includegraphics[width=35mm]{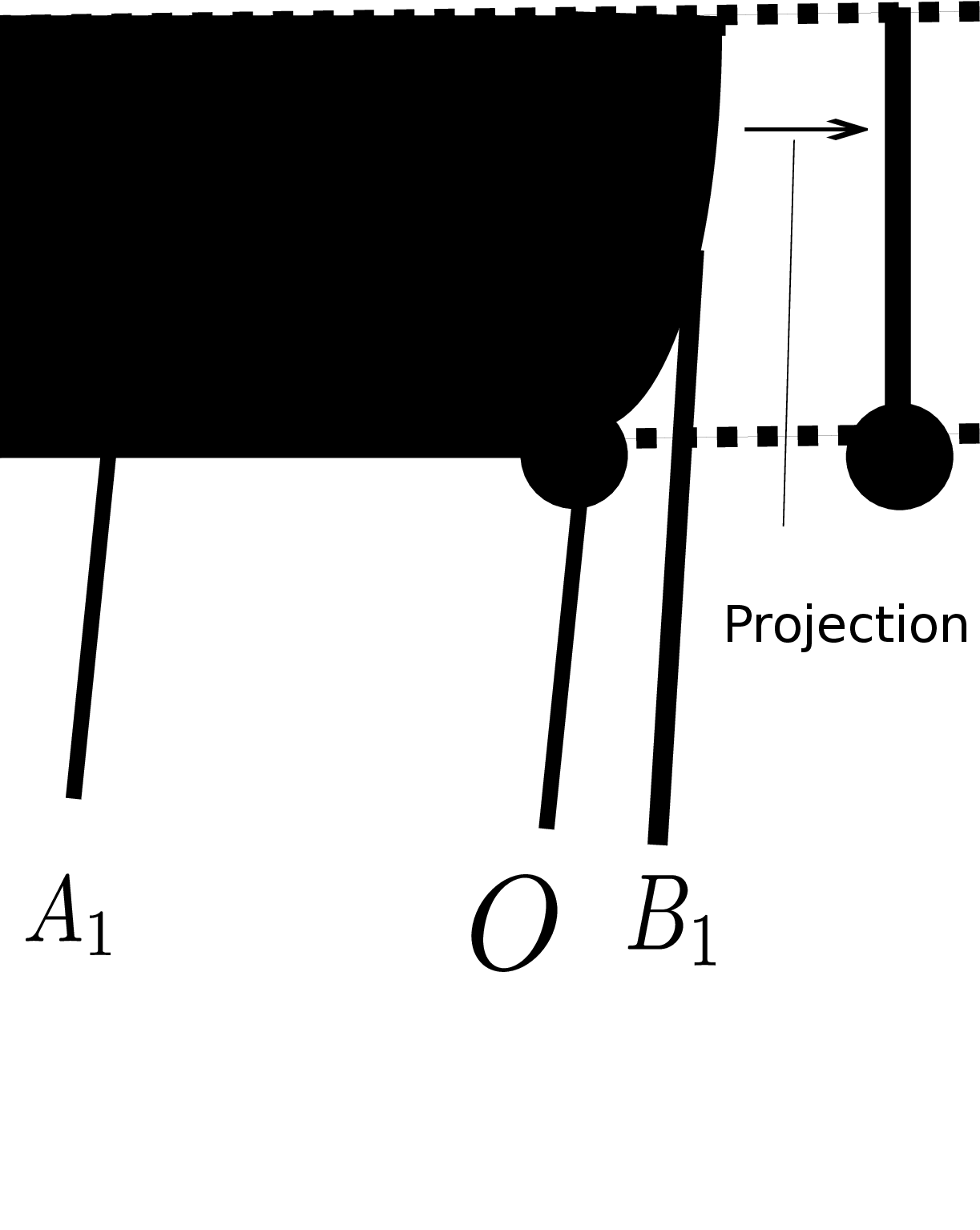}
\caption{The image of a smooth map $\tilde{F_1}$ into the plane and the composition of the map with the projection $p$. The dot is the singular value also in this case.}
\label{fig:8}
\end{figure}

As a result, for each vertex at which the good function has a local extremum, we obtain a local function on an orientable surface satisfying the first, second, third and fifth properties in the statement. This completes STEP 2. \\
 \\
\noindent STEP 3 Completing the construction. \\

Last, we construct functions around remaining parts. However, we can do this in a way similar to the proof of Theorem \ref{thm:2}. This completes the proof.

\end{proof}
\begin{Thm}
\label{thm:5}
Let $n$ be a positive integer.
As Theorem {\rm \ref{thm:3}}, the {\rm (}$n+1${\rm )}-dimensional version of Theorem {\rm \ref{thm:4}} holds where we take the $n$-dimensional unit sphere $S^n$ and the Euclidean space ${\mathbb{R}}^n$ instead of a circle and a line respectively as connected components of preimages of regular values.
\end{Thm}
We can prove this similarly to Theorem \ref{thm:4}. We only present remarks and rigorous proofs are left to readers.
\begin{proof}[Remarks on the proof]
In the discussions of considering $F_0$, $F$, $F^{\prime}$, $F_1$ and their extensions in the proof of Theorem \ref{thm:4}, we take a standard ($n-1$)-dimensional sphere instead of the two-point set as the preimage of each regular value. For a local Morse function, we take a natural height function on an $n$-dimensional standard disk instead of a Morse function on a line whose graph is a parabola.
\end{proof}

\section{A further result.}
\label{sec:4}
\begin{Thm}
\label{thm:6}
Let $n>1$ be an integer.
Let $G$ be a connected graph having at least one edge. Assume that a good function $h$ is given. Assume also that a non-negative integer is assigned to each edge satisfying the following condition by a map $l${\rm :} if an edge $e$ contains a vertex of degree $1$, then $l(e)=0,1,2$.

In this situation there exist an {\rm (}$n+1${\rm )}-dimensional connected and orientable manifold $M$ and a smooth function $f:M \rightarrow \mathbb{R}$ satisfying the following five properties.
\begin{enumerate}
\item The Reeb graph $W_f$ is isomorphic to $G$ and an isomorphism $\phi:W_f \rightarrow G$ exists.
\item If we consider the natural quotient map onto the graph and for each point that is not a vertex and that is in an edge $e$ satisfying $l(e)=q \geq 0$, then the preimage is diffeomorphic to a manifold obtained by removing $q$ standard disks disjointly and smoothly embedded into an $n$-dimensional standard sphere {\rm (}$q=0$ means that the resulting manifold is diffeomorphic to $S^n$ and $q=1$ means that the resulting manifold is diffeomorphic to $\mathbb{R}${\rm )}. 
\item For a point $p \in M$ mapped by the quotient map $q_f$ to a vertex $v$, $f(p)=h(\phi(v))$. 
\item At singular points at which the function $f$ does not have local extrema, except finitely many ones, the local functions are Morse functions or Morse-Bott functions.
\item At singular points at which the function $f$ has local extrema, except finitely many ones, the local functions are Morse functions, Morse-Bott functions, or represented as the compositions of two Morse functions or compositions of Morse-Bott functions with Morse functions.
\end{enumerate}
\end{Thm}
\begin{proof}
\noindent STEP 1 Construction around a vertex at which $h$ does not have a local extremum. \\

We first construct a local Morse function as in Theorem \ref{thm:2} so that for each edge containing the vertex, the second property of the five properties holds where $0$ is assigned to each edge. This is also a main ingredient of a main theorem of \cite{michalak} and explained in \cite{kitazawa2}.
For the vertex, we may assume that the third property of the five properties holds. We can also obtain the function so that locally the Reeb graph is isomorphic to $G$: as is explained in the proof of Theorem \ref{thm:2}, there exists a homeomorphism from the Reeb space onto a small regular neighborhood of the vertex in the given graph mapping a vertex (the point representing the preimage containing some singular points) to a vertex of the given graph.

This argument is completed by locally using functions (on ($n+1$)-dimensional manifolds) in CASE 1-C in the proof of Theorem \ref{thm:4} or \ref{thm:5} instead of trivial smooth bundles whose fiber is a standard $n$-dimensional disk over the image, which is an interval, suitably. This function changes manifolds diffeomorphic to $D^{n}$ in preimages into ones diffeomorphic to $S^{n-1} \times [0,1)$ when the value of the function increases or decreases. This change increases the numbers of $n$-dimensional standard disks disjointly and smoothly embedded into an $n$-dimensional standard sphere to remove to obtain manifolds of the connected components of preimages of regular values by one. For each edge $e$ containing the vertex, we choose $l(e)$ trivial smooth bundles whose fiber is a standard $n$-dimensional disk over the interval suitably. We can take the (total spaces of) bundles disjointly. We replace them by suitable local functions in a suitable way.

Thus we can obtain a desired local function. At each singular point of the obtained function, except finitely many ones, this is a Morse function or a Morse-Bott function. \\
 \\
\noindent STEP 2 Construction around a vertex at which $h$ has a local extremum. \\

If the vertex is of degree greater than $1$, then we can apply an argument similar to that of the proof of Theorem \ref{thm:2} (Theorem \ref{thm:4}) using the local function.

Consider a case where the vertex is of degree $1$. The case where $l(e)=2$ is assigned to the edge $e$ containing the vertex is the only one new case to consider. We consider a natural height function of an $n$-dimensional standard disk and the product of this and the identity map on a line so that the image is $\mathbb{R} \times [0,1]$ and that the singular value set is $\mathbb{R} \times \{0\}$ ($\mathbb{R} \times \{1\}$). 
Last we compose the projection $p(x,y):=y$ and scale the resulting function. Thus we obtain a desired local function.\\
 \\
This with arguments similar to ones in the presented proofs completes the proof.
\end{proof}
\section{Acknowledgement.}
\label{sec:5}
The author is a member of the project JSPS KAKENHI Grant Number JP17H06128 "Innovative research of geometric topology and singularities of differentiable mappings" (Principal Investigator: Osamu Saeki) and this work is supported by this project. The author would like to thank Osamu Saeki, colleagues supporting this project and Irina Gelbukh for motivating the author to produce the present study through discussions on \cite{kitazawa2}. Last, the author would like to thank anonymous referees.

We also declare that data supporting our present study essentially are all in the present paper.

\end{document}